\documentclass[11pt]{amsart}
\usepackage[T1]{fontenc}
\usepackage{amsfonts}
\usepackage{verbatim}
\usepackage{amsmath, amsthm, amssymb, enumerate,color}
\usepackage[comma,sort&compress]{natbib}
\usepackage{rotating}
\usepackage{subfigure}

\numberwithin{equation}{section}
\allowdisplaybreaks 

\usepackage{graphicx,subfigure}

\allowdisplaybreaks 
\numberwithin{equation}{section}
\newtheorem{prop}{Proposition}

\newtheorem{theorem}[prop]{Theorem}

\theoremstyle{definition}

\newcommand{\vep}{\varepsilon}

        \newcommand{\vrightarrow}{\,{\buildrel v \over \rightarrow}\,}

\newcommand{\alphain}{\alpha_{\text in}}
\newcommand{\alphaout}{\alpha_{\text out}}
\newcommand{\deltain}{\delta_{\text in}}
\newcommand{\deltaout}{\delta_{\text out}}

\def\bx{\boldsymbol x}
\def\by{\boldsymbol y}
\def\bu{\boldsymbol u}
\def\bv{\boldsymbol v}

\def\bu{\boldsymbol u}
\def\bone{\boldsymbol 1}

\def\bs{\boldsymbol s}

\def\E{\mathbb{E}}
\def\bzero{\boldsymbol 0}

\def\bx{\boldsymbol x}

\def\by{\boldsymbol y}
\def\bu{\boldsymbol u}
\def\bv{\boldsymbol v}

\def\binfty{\boldsymbol \infty}

\def\bgamma{\boldsymbol \gamma}
\def\bb{\boldsymbol b}
\def\blambda{\boldsymbol \lambda}

\def\Rplus{\mathbb{R}_+}

\usepackage{pdfsync,url}

\begin{document}

\title[Tauberian Theory and Preferential Attachment]{Tauberian Theory for
  Multivariate Regularly Varying Distributions with Application to
  Preferential Attachment Networks}

\author[S. Resnick]{Sidney Resnick}
\address{School of Operations Research and Information Engineering\\
and Department of Statistical Science \\
Cornell University \\
Ithaca, NY 14853}
\email{sir1@cornell.edu}
\author[G. Samorodnitsky]{Gennady Samorodnitsky}
\address{School of Operations Research and Information Engineering\\
and Department of Statistical Science \\
Cornell University \\
Ithaca, NY 14853}
\email{gs18@cornell.edu}

\thanks{This research was supported by the ARO
MURI grant  W911NF-12-10385 to Cornell University.}

%\vspace{.5ex}}

%\newtheorem{prop}{Proposition}
\setcounter{lemma}{0}

%\maketitle

\begin{abstract}
Abel-Tauberian theorems relate power
law behavior of distributions and their transforms. We formulate and
prove a multivariate version for non-standard regularly varying
measures on $\mathbb{R}_+^p$ and then apply it to 
prove that the joint distribution of  in- and out-degree in a directed edge 
preferential attachement model has  jointly regularly varying
tails. 
\end{abstract} 

\maketitle
\section{Introduction}\label{sec: Intro}
This paper has two themes: (i) Abel-Tauberian theorems relate power
law behavior of distributions and their transforms. (ii) Such
Abel-Tauberian theorems can be used to study power law behavior
of in- and out-degree of directed edge preferential attachment 
network models.

Abel-Tauberian theorems  relate
 regular variation of infinite Radon measures $U(\bx) =U([\bzero,\bx])$ on
$\mathbb{R}_+^p$ 
to  regular variation of their Laplace transforms $$\hat
U(\bone/\bx)=\int_{\mathbb{R}_+^p}
e^{-\sum_{i=1}^p s_i/x_i} U(d\bs), \qquad \bx>\bzero.$$ 
In one dimension when $p=1$, such theorems provide standard tools for 
obtaining asymptotic power law tails for
cases when a probability description is more easily specified by 
the transform rather than the distribution.  Often the transform
rather than the distribution
is accessible  as a solution to difference or
recursive relations. Application areas include queueing theory,
branching processes, insurance modeling and network
analysis. Standard references covering the essentials in one dimension
are \cite{bingham:goldie:teugels:1987, feller:1971}. 
Transform theory when $p>1$  
for the standard case of regular variation are considered
in \cite{stam:1977, resnick:2007, resnick:1991,
  stadtmuller:trautner:1979,
  stadtmuller:trautner:1981, stadtmuller:1981, yakimiv:2005}.  In this
paper we consider an Abel-Tauberian theorem for the non-standard case
of regular variation where scaling functions for different components
have different tail indices.

Preferential attachment is an important mechanism for describing
growth of directed networks where a new node attaches to an existing
node or new edges are created according to probabilistic  postulates that take into
account the in- and out-degree of the existing nodes. We consider
models studied by \cite{bollobas:borgs:chayes:riordan:2003} and
\cite{krapivsky:redner:2001}. Based on  solutions to difference
equations, \cite{resnick:samorodnitsky:towsley:davis:willis:wan:2014} derived the joint generating function of limiting
frequencies for in-degree and out-degree. In this paper we explain how
the joint non-standard regular variation of in- and out-degree can be
obtained from the joint generating function  using
Abel-Tauberian theory.

This paper is organized as follows. We start with a brief summary of
multivariate regular variation  of measures in Section
\ref{sec:mult.reg.var} to establish notation and basic concepts.
Section \ref{sec:Tauberian.thm} gives the Abel-Tauberian theorem for
measures which are non-standard regularly varying. In Section 
\ref{sec:InOutRegVar}, we apply the Tauberian theory to study the
multivariate power laws of in- and out-degree in the preferential
attachment model. Section \ref{subsec:moddes} includes a detailed
description of the preferential attachment model, Section
\ref{subsec:results} summarizes known results about the joint
generating function of in- and out-degree and Section
\ref{sec:InOutRegVar} applies the Tauberian theory to obtain the joint
power law behavior of in- and out-degree.

\section{Multivariate regular variation} \label{sec:mult.reg.var} 
We briefly review the basic concepts of multivariate regular
variation which forms the mathematical framework for multivariate
heavy tails. We emphasize two dimensions since this is the
context for the application to in- and out-degree but generally the
extension to $p\geq 2$ dimensions is clear.
 See \cite{resnick:2007} for more detail.

A random vector $(X,Y)\geq \bzero$ has a distribution that is 
non-standard regularly varying  if there exist {\it
  scaling functions\/} $a(t)\uparrow \infty$ and  $b(t)\uparrow
\infty$ and a non-zero limit measure $\nu(\cdot)$ called the {\it limit or tail
measure\/} such that as $t\to\infty$,
\begin{equation}\label{e:defreg}
tP\bigl[\bigl(X/a(t),Y/b(t)\bigr) \in \cdot \,\bigr] \vrightarrow  \nu(\cdot)
\end{equation}
where ``$\vrightarrow $'' denotes vague convergence of measures in
$M_+([0,\infty]^2\setminus \{\bzero\})=M_+(\E)$, the space of Radon
measures on $\E$. The scaling functions will be regularly varying and
we assume their indices are positive and therefore, without loss of
generality, we may suppose $a(t)$ and $b(t)$ are continuous and
strictly increasing. The phrasing in \eqref{e:defreg} implies the
marginal distributions have regularly varying tails.

In case $a(t)=b(t)$, $(X,Y)$ has a distribution with {\it standard\/}
  regularly varying tails. Given a vector with a distribution which is
  non-standard regularly varying, there are at least two methods for
  standardizing the vector so that the transformed vector has standard
  regular variation \citep[Section 9.2.3]{resnick:2007}. The simplest is the power method which is
  justified when the scaling functions are power functions:
$$a(t)=t^{1/\gamma_1},\quad b(t)=t^{1/\gamma_2},\quad \gamma_i>0, \,i=1,2.$$
For instance with $c=\gamma_1/\gamma_2$,
\begin{equation}\label{e:standard}
tP\bigl[ \bigl( X^c/t^{1/\gamma_2}, Y/t^{1/\gamma_2}\bigr) \in \cdot \,]
  \vrightarrow  \tilde \nu(\cdot),
\end{equation}
where if $T(x,y)=(x^c,y)$, then $\tilde \nu=\nu\circ T^{-1}.$ Since
the two scaling functions in \eqref{e:standard} are the same, the regular
variation is now standard. The measure $\tilde \nu$ will  have a
scaling property and if the coordinate system is changed properly,
$\tilde \nu$ will disintegrate to a product; for example the polar
coordinate transform  is one such coordinate system change achieving
the disintegration into a product and this provides access to an
angular measure that is one way to describe the
asymptotic dependence structure of the standardized $(X,Y)$.

The non-standard regular variation of Radon measures is defined in
\eqref{eq:b} below. 

\subsection{Miscellaneous notation.}
\label{subsec:notation}

Here is a  notation and concept summary.
$$ 
\begin{array}{llll}
RV_\beta & \text{Regularly varying functions with index $\beta>0$. We
pick versions}\\
&\text{of such functions that are continuous and strictly
  increasing.}\\[2mm]
M_+(\E) & \text{Radon measures on  $\E:=[0,\infty]^p \setminus \{\bzero\}$ metrized by
  vague convergence.}\\[2mm]
M_+(\Rplus^p) & \text{Radon measures on $\Rplus^p$ metrized by
  vague convergence.}\\[2mm]
\stackrel{v}{\to} & \text{Vague convergence in $M_+(\Rplus^p)$ or
  $M_+(\E)$ as appropriate.}\\[2mm]
\bx & \bx=(x_1,\dots,x_p).\\[2mm]
\blambda \bx & (\lambda_1x_1,\dots,\lambda_p x_p).\\[2mm]
\blambda'\bx & \sum_{i=1}^p \lambda_i x_i.\\[2mm]
\bone & \bone=(1,\dots,1).\\[2mm]
\bzero & \bzero=(0,\dots,0).\\[2mm]
\hat U & \text{Laplace transform of a measure $U$; $\hat
  U(\blambda)=\int_{\mathbb{R}_+^p} \exp\{-\blambda'\bx\}U(d\bx)$.}\\[2mm]
\aleph & \aleph =\{\bx \in \Rplus^p: \|\bx \|=1\}, \text{ the unit
  sphere in $\mathbb{R}_+^p$ for some
  norm $\|\cdot\|$.}\\
\end{array}
$$

In general vectors are denoted by bold letters,
eg. $\bx=(x_1,\dots,x_p)$.  Operations on vectors, unless noted
otherwise, should be interpreted componentwise. Thus, $\blambda \bx =
(\lambda_1x_1,\dots,\dots,\lambda_p x_p)$ but
(as noted)
$\blambda'\bx = \sum_{i=1}^p \lambda_i x_i.$ Also
$[\bzero,\bx]=\{(u_1,\dots,u_p): 0\leq u_i \leq x_i, \,i=1,\dots,p\}$.

\section{A Tauberian theorem for nonstandard regular variation}
\label{sec:Tauberian.thm} 

In this section we give an Abel-Tauberian theorem which relates
non-standard regular variation of a Radon measure $U(\bx)$ on $\mathbb{R}_+^p$
to non-standard regular variation of the Laplace transform $\hat
U(\bone/\bx)$. Versions in the standard case when $p>1$ are considered
in \cite{stam:1977, resnick:2007, resnick:1991,
  stadtmuller:trautner:1979,
  stadtmuller:trautner:1981, stadtmuller:1981, yakimiv:2005}.

\subsection{Assumptions}\label{subsubsec:basic}For $p\geq 1$, suppose $U$ is a measure on $\mathbb{R}_+^p$ satisfying 
\begin{equation}\label{eq:aa}
\hat U(\blambda):= \int_{\mathbb{R}_+^p} e^{-\blambda '  \bx} U(d\bx)<\infty, \quad
\blambda >\bzero.
\end{equation}
This implies $U$
is Radon on $\mathbb{R}^p_+$ since for $\blambda >\bzero$, and $\by>\bzero$,
\begin{align*}
\infty > & \int_{\mathbb{R}_+^p} e^{-\blambda '  \bx} U(d\bx)
\geq  \int_{\mathbb{R}_+^p} e^{-\blambda '  \bx}
1_{[\bzero,\by]}(\bx) U(d\bx)\\
\geq & e^{-\blambda '\by } \int 1_{[\bzero,\by]}(\bx)
U(d\bx)=e^{-\blambda '\by } U([\bzero,\by]).
\end{align*}
So $U(\by):=U([\bzero,\by]) <\infty$ for $\by>\bzero$ and therefore
 $U \in M_+(\mathbb{R}_+^p)$.

For $i=1,\dots,p$, assume 
\begin{equation}\label{eq:bRegVar}
b_i(t) \in RV_{1/\gamma_i},\, \gamma_i >0, \quad i=1,\dots,p.
\end{equation}
 Write
$\bb(t) =(b_1(t),\dots,b_p(t))$  and $\bgamma
=(\gamma_1,\ldots,\gamma_p)$. Set 
\begin{equation}\label{eq:defUt}
U_t(\bx)=\frac 1t U(\bb (t) \bx).
\end{equation}

\subsection{Regular variation of the measure implies
  regular variation of the
  transform.}\label{subsub:measureImpliesTransform}
For this section assume $U$ satisfies
\eqref{eq:aa} and  $U_t$ is defined in \eqref{eq:defUt}. The scaling
functions $b_i(t)$ satisfy
\eqref{eq:bRegVar}. The non-standard regular variation assumption for
$U$ is that there 
exist $U_\infty \in M_+(\mathbb{R}^p_+) $, $U_\infty \not \equiv 0$, such that 
\begin{equation}\label{eq:b}
U_t \stackrel{v}{\to} U_\infty,\qquad \text{ in } M_+(\mathbb{R}^p_+) .
\end{equation}
If we can choose the scaling functions $(b_i, \, i=1,\ldots,p)$ to be
identical, then the regular variation is standard. 

\subsubsection{Consequences}\label{subsubsec:conseq}
The assumptions have
consequences needed for further work.

\subsubsection*{1. Continuous convergence:}  The convergence in \eqref{eq:b} is {\it
    continuous convergence\/} on $(\bzero,\binfty):=(0,\infty)^p$;
  that is, if as $t\to\infty$, $\bx(t) \to \bx(\infty ) \in (\bzero,\binfty)$, then 
\begin{equation}\label{eq:contconv}
U_t(\bx (t)) \to U_\infty (\bx (\binfty)),\quad (t\to\infty),
\end{equation}
provided $\bx (\infty) $ is a continuity point of $U_\infty(\bx) $.
This is a monotonicity argument: If $\bx(\infty)$ and $\bx (\infty)
+\epsilon \bone $ are continuity points of $U_\infty(\bx)$, then
\begin{align*}
\limsup_{t\to\infty} U_t(\bx(t)) \leq &\lim_{t\to\infty} U_t (\bx
(\infty) +\epsilon \bone)\\
=& U_\infty( \bx(\infty) +\epsilon \bone),\\
\intertext{and  letting $\epsilon \downarrow 0$ in such a way that $\bx
  (\infty)+\epsilon \bone $ are continuity points of $U_\infty (\bx)$
  yields}
\limsup_{t\to\infty} U_t(\bx(t)) \leq & U_\infty (\bx (\infty)).
\end{align*}
A reverse inequality is obtained similarly. A consequence of the
continuous convergence is the scaling property: for $c>0$  
\begin{equation} \label{e:scling.genmap}
U_\infty\circ T_c^{-1} = cU_\infty,
\end{equation}
where $T_c:\, \mathbb{R}_+^p\to \mathbb{R}_+^p$ is defined by
$T_c\, \bx=c^{-1/\bgamma}\bx$. It is enough to check that for $\bx>0$ 
\begin{equation} \label{e:scling.gen}
U_\infty \bigl( c^{\bone/\bgamma}\bx\bigr)  =
cU_\infty(\bx)\,.
\end{equation}
Indeed, 
\begin{align*}
U_\infty \bigl( c^{\bone/\bgamma}\bx\bigr) = 
&\lim_{t\to\infty} \frac 1t U\bigl(b_1(t)
c^{1/\gamma_1}x_1,\dots, b_p(t)
c^{1/\gamma_p}x_p \bigr)\\
\intertext{and by continuous convergence, this is}
=&\lim_{t\to\infty} c  \frac {1}{ct}
U\Bigl(b_1(ct)\Bigl(\frac{b_1(t)}{b_1(ct)}c^{1/\gamma_1}\Bigr)x_1,\dots, 
b_p(ct)\Bigl(\frac{b_p(t)}{b_p(ct)}c^{1/\gamma_p}\Bigr)x_p 
\Bigr)\\
=& cU_\infty (\bx).
\end{align*}
The scaling property implies, in particular, that all points $\bx$ are
continuity points of $U_\infty$.

\subsubsection*{2. Laplace transform of $U_\infty$ exists:}  Let
$i_\ast\in \{ 1,\ldots, p\}$ be such that $\gamma_{i_\ast}\geq
\gamma_i$ for all $i\in \{ 1,\ldots, p\}$. It follows from the scaling
property \eqref{e:scling.genmap} that for any $a>0$
$$
U_\infty\left( \left\{ \bx:\, \sum_{i=1}^p x_i\leq a\right\}\right)
\leq a^{\gamma_{i_\ast}} U_\infty\left( \left\{ \bx:\, \sum_{i=1}^p
    x_i\leq 1\right\}\right)\,.
$$
Therefore, for
$\blambda>\bzero$, 
\begin{equation} \label{eq:finite}
\hat U_\infty (\blambda) \leq \int_{\mathbb{R}_+^p}
e^{-\min_i \lambda_i \sum_ix_i}\, U_\infty (d\bx)
\end{equation}
$$
\leq U_\infty\left( \left\{ \bx:\, \sum_{i=1}^p
    x_i\leq 1\right\}\right) \int_0^\infty e^{-(\min_i \lambda_i) x}\, 
\gamma_{i_\ast} x^{\gamma_{i_\ast}-1}\, dx<\infty\,.
$$

\subsubsection{The result.}\label{subsubsec:measImpliesTransform}
This section requires a regularity condition: for any $\bx>\bzero$,
\begin{equation}
\lim_{y\to\infty} \limsup_{t\to\infty}\int_{\cup_{i=1}^p[v_i>y]}  e^{-\sum_{i=1}^p v_i/x_i} U_t(d\bv)
=0.\label{eq:7early}
\end{equation}

\begin{prop} Assume \eqref{eq:bRegVar} and  suppose that $U$ satisfies
  the non-standard 
  regular variation condition\eqref{eq:b}. Then the Laplace transforms
$\hat U(\bone/\bx)$ and $\hat U_\infty (\bone/\bx)$ are distribution
functions of Radon measures on $\mathbb{R}_+^p $ and these measures
inherit the non-standard regular variation: for $\bx>0$
\begin{equation} \label{e:conv.tr}
\frac 1t \hat U\bigl(\bone/(\bb(t)\bx )\bigr) \to \hat U_\infty
(\bone/\bx),
\end{equation}
provided \eqref{eq:7early} also holds.
\end{prop}

\begin{proof}
Let $E_1,\dots,E_p$ be iid standard exponentially
distributed random variables so that
$$\mathcal{F}=\Bigl(
\frac{1}{E_1},\dots,\frac{1}{E_p} 
\Bigr)
$$
are iid standard Frech\'et random variables with marginal distribution
$$P[1/E_1 \leq x]=e^{-1/x},\quad x>0.$$
From \eqref{eq:b} we get
\begin{equation}\label{eq:5}
P[\mathcal{F} \in \cdot\,]\times U_t \stackrel{v}{\to} P[\mathcal{F}
\in \cdot \,] \times U_\infty,
\end{equation}
in $M_+\bigl([0,\infty]^p \times \mathbb{R}_+^p\bigr). $  Define $h:[0,\infty]^p
\times \mathbb{R}_+^p \mapsto [0,\infty]^p \times \mathbb{R}_+^p$
by 
$$ h(\bx,\by) =(\bx \by,\by),$$
where $\bx \by=(x_iy_i,i=1,\dots,p)$ is componentwise multiplication,
and we set $0\cdot\infty=0$. The map $h$ satisfies the 
compactness condition of \cite[Proposition 5.5]{resnick:2007}: Suppose 
$A\subset [0,\infty]^p\times \mathbb{R}_+^p$ satisfies the condition that there exists
$M>0$ such that 
$$(\bx,\by) \in A \quad \text{ implies } \bigvee_{i=1}^p y_i \leq M.$$
Then
$$h^{-1}(A)=\{(\bu,\bv): (\bu\bv,\bv) \in A\}$$
satisfies 
$$(\bx,\by) \in h^{-1}(A)  \quad \text{ implies } \bigvee_{i=1}^p y_i
\leq M.$$ Thus if $A$ is relatively compact, so is $h^{-1}(A)$.
Therefore \eqref{eq:5} and \cite[Proposition 5.5]{resnick:2007} imply
\begin{equation}\label{eq:6}
\bigl( P[\mathcal{F} \in \cdot\,]\times U_t \bigr)\circ h^{-1}
\stackrel{v}{\to} \bigl( P[\mathcal{F}
\in \cdot \,] \times U_\infty \bigr)\circ h^{-1}, \text{ in
}M_+([0,\infty]^p\times \Rplus^p).
\end{equation}
Evaluate the left side of \eqref{eq:6} on the relatively compact
 set $[\bzero, \bx]\times [\bzero,y\bone ] $
(assuming $\bx>\bzero$ and $y>0 $ are chosen to make this is a continuity set of the
limit measure) and we get,
\begin{align}
\bigl( P[\mathcal{F} \in \cdot\,]&\times U_t \bigr)\circ h^{-1} \bigl( [\bzero, \bx]\times
[\bzero,y\bone ] \bigr)
=\iint_{\{ (\bu,\bv): \bu\bv \leq \bx, \bv \leq y\bone \}}
P[\mathcal{F} \in d\bu\,] U_t(d\bv)  \nonumber  \\
=&\int_{\bv\leq y\bone } \int_{\bu \leq \bx/\bv } P[\mathcal{F} \in
d\bu\,] U_t(d\bv) 
=\int_{\bv\leq y\bone } \prod_{i=1}^p e^{-v_i/x_i} U_t(d\bv)\nonumber\\
=&\int_{\bv\leq y\bone }  e^{-\sum_{i=1}^p v_i/x_i} U_t(d\bv) \label{eq:hatU}
\\
\intertext{and applying \eqref{eq:6} we conclude that as
  $t\to\infty$ this converges to}
\to & \int_{\bv\leq y\bone }  e^{-\sum_{i=1}^p v_i/x_i}
U_\infty(d\bv).\label{eq:notthereyet} \\
\intertext{
Now let $y\to\infty$ and apply monotone convergence to get
the integral in \eqref{eq:notthereyet} 
 to converge to  }
\to & \int_{\mathbb{R}_+^p}  e^{-\sum_{i=1}^p v_i/x_i} U_\infty(d\bv)=:
\hat U_\infty (\bone/\bx).\nonumber 
\end{align}
So to show for $\bx>\bzero$ that 
\begin{equation}\label{eq:taub}
\hat U_t(\bone/\bx) :=\int_{\mathbb{R}_+^p}  e^{-\sum_{i=1}^p v_i/x_i}
U_t(d\bv) =\frac 1t \hat U(\bone/(\bb(t)\bx)) \to \hat U_\infty
(\bone/\bx),
\end{equation}
we must verify that 
\begin{align*}
\lim_{y \to \binfty} \limsup_{t\to\infty}
\Bigl|
\int_{\bv\leq y\bone }  &e^{-\sum_{i=1}^p v_i/x_i} U_t(d\bv)-
\int_{\mathbb{R}_+^p}  e^{-\sum_{i=1}^p v_i/x_i} U_t(d\bv)  
\Bigr| =0,
\end{align*}
which is \eqref{eq:7early}.

The statement that $\hat U(\bone/\bx)$ is a distribution function of a
Radon measure follows from \eqref{eq:aa} since, as in \eqref{eq:hatU}, 
$$\infty>\hat U(\bone/\bx)=\lim_{y\to\infty}
\bigl( P[\mathcal{F} \in \cdot\,]\times U \bigr) \circ h^{-1}\bigl(  [\bzero, \bx]\times
[\bzero,y\bone ]\bigr),$$
a limit of the distribution functions of a
Radon measures. The statement about $\hat
U_\infty(\bone/\bx)$ follows similarly using the fact that $\hat
U_\infty(\blambda)<\infty$ for $\blambda >0$ by \eqref{eq:finite}.
\end{proof}

Rather than checking condition \eqref{eq:7early} directly, it may sometimes
be easier to
verify the following sufficient condition: for every $1\leq i\leq p$, suppose
\begin{align}
&U_i(x) =U(\mathbb{R}_+\times \dots \times [0,x]\times \mathbb{R}_+
\times \dots \times \mathbb{R}_+)<\infty,\label{eq:7a}\\
\intertext{and}
&\lim_{t\to\infty} \frac{U_i(b_i(t)x)}{t} =x^{\gamma_i},\quad
x>0.\label{eq:7b}
\end{align}
To see why these conditions are sufficient for \eqref{eq:7early}, 
dominate the integral in \eqref{eq:7early} by 
$$
\sum_{i=1}^p\int_{[v_i>y]}  e^{-\sum_{i=1}^p v_i/x_i} U_t(d\bv) 
$$ and focus, for simplicity, on the integral with $i=1$ which can be 
written as
\begin{align*}
\int_{[v_1>y]} &\Bigl[ \prod_{l=1}^p \int_{s_l>v_l} \frac{1}{x_l} e^{-s_l/x_l}
ds_l \Bigr] U_t(d\bv) \\
=&\int_{s_1>y} \Bigl( \int_{\substack{y<v_1\leq s_1\\s_l\geq v_l;
    l=2,\dots,p}}
U_t(d\bv) \Bigr)   \prod_{l=1}^p \frac{1}{x_l} e^{-s_l/x_l}ds_1\dots
ds_p\\
=& \int_{\bs \in (y,\infty)\times \mathbb{R}_+^{p-1} }
U_t \bigl( (y,s_1]\times [0,s_2]\times \dots \times[0,s_p]
\bigr)  \prod_{l=1}^p \frac{1}{x_l} e^{-s_l/x_l}ds_1\dots
ds_p\\
\leq & \int_{\bs \in (y,\infty)\times \mathbb{R}_+^{p-1} }
U_t (\bs)
  \prod_{l=1}^p \frac{1}{x_l} e^{-s_l/x_l}ds_1\dots
ds_p\\
\leq & \int_{y}^\infty U_t([0,s_1]\times \mathbb{R}_+^{p-1} ) 
 \frac{1}{x_1} e^{-s_1/x_1}ds_1\\
=&\int_y^\infty \frac{U_1(b_1(t)s_1)}{t}  \frac{1}{x_1}
e^{-s_1/x_1}ds_1\\
\intertext{and by an application of the Potter bounds, for given
  $\delta>0$ and large  enough $t$ and $y>1$, the previous expression
  is bounded by}
\leq & \int_y^\infty (const)s^{\gamma_1+\delta}  \frac{1}{x_1}
e^{-s_1/x_1}ds_1 \to 0, \quad (y\to\infty).
\end{align*}

\subsection{Regular variation of the transform
implies
  regular variation of the
measure.}\label{subsub:TransformImpliesMeasure}
In this section we assume \eqref{eq:aa}, \eqref{eq:bRegVar},
\eqref{eq:7early}
and
additionally assume there exists a finite-valued function 
$\hat U_\infty$  such that for $\bx>\bzero$,
\begin{equation}\label{eq:8}
\frac 1t \hat U (\bone/(\bb (t) \bx))=\bigl( P[\mathcal{F} \in \cdot
\,]\times U_t \bigr)\circ h^{-1} \bigl([\bzero,\bx]\times
\mathbb{R}_+^p\bigr)
\to \hat U_\infty (\bone/\bx).
\end{equation} 

We claim that $\{U_t\}$ is a tight family of measures on
$\mathbb{R}^p_+$.  It suffices to show that for any $M>0$
$$\sup_{t\geq 1} U_t[\bzero,M\bone]<\infty.$$
Given $\epsilon>0$, there exists $\delta>0$ such that $P\bigl[\mathcal{F}
\in [\delta \bone, \delta^{-1}\bone]\bigr]\geq 1-\epsilon$. 
For $\bx>\bzero$, 
\begin{align*}
\hat U_t(\bone/\bx)=& 
\bigl( P[\mathcal{F} \in \cdot \,]\times U_t\bigr) \circ h^{-1} ([\bzero,
\bx]\times \mathbb{R}_+^p)\\
=& \int_{\{(\bu,\bv):\bu\bv\leq \bx \}} P[\mathcal{F} \in d\bu\,] U_t  (d\bv)
\geq  \int_{\substack{\bu\bv \leq \bx\\ \bu \in [\delta \bone,
    \delta^{-1}\bone]}}  P[\mathcal{F} \in d\bu\,] U_t  (d\bv)\\
=& \int_{ \bu \in [\delta \bone,  \delta^{-1}\bone]}  U_t  (
\bx/\bu)P[\mathcal{F} \in d\bu\,] \\
\geq & U_t(\bx/\delta^{-1}) P\bigl[\mathcal{F} \in  [\delta \bone,
\delta^{-1}\bone] \bigr] \geq  U_t(\bx/\delta^{-1}) (1-\epsilon).
\end{align*}
Set $\bx=\delta^{-1} M \bone$ and then 
$$\sup_{t\geq 1} U_t(M\bone) \leq \frac{1}{1-\epsilon}\sup_{t\geq 1} \hat
U_t(1/(\delta^{-1} M \bone))<\infty$$
by convergence in \eqref{eq:8}.

Suppose $\{U_{t_n}\}$ is a convergent subsequence, say $U_{t_n} \to L$
in $M_+(\mathbb{R}_+^p)$. Since we assume \eqref{eq:7early} holds, the
mechanics of Section  \ref{subsubsec:measImpliesTransform}
give for $\bx>\bzero$,
\begin{equation}\label{eq:subseqLimit}
 \hat U_{t_n} (\bone/\bx) \to \hat L(\bone/\bx)<\infty,\quad
(t_n\to\infty)\end{equation}
at continuity points of the limit. From \eqref{eq:8} we get
$\hat L=\hat U_\infty$. If there are two subsequential limits
$L_1,L_2$ of $\{U_t\}$ then $\hat L_1 =\hat L_2=\hat U_\infty$ and so
$\{U_t\}$ converges in $M_+(\mathbb{R}_+^p) $ to some $U_\infty$ with
transform $\hat U_\infty$.

We summarize:

\begin{prop}\label{prop:2} Suppose $U \in M_+(\mathbb{R}_+^p)$   and
  let   \eqref{eq:aa},  
\eqref{eq:bRegVar}, \eqref{eq:7early} hold. If there exists a finite-valued function 
$\hat U_\infty$ such that \eqref{eq:8} holds,  then \eqref{eq:b}
holds for some measure $U_\infty \in M_+(\mathbb{R}_+^p)$ whose
Laplace transform is $\hat U_\infty$. Moreover, 
$$
U_t(\bx)=\frac 1t U(\bb (t) \bx) \to U_\infty (\bx),\quad
(t\to\infty)$$
for all $\bx$. 
\end{prop}

\section{Application to preferential attachment network models.} \label{sec:genfuncion} 
\subsection{Model description.}\label{subsec:moddes}
The directed edge  preferential attachement model 
studied by \cite{krapivsky:redner:2001}
and \cite{bollobas:borgs:chayes:riordan:2003} 
is a model for a growing directed random graph. The dynamics of the
model are as follows. Choose nonnegative real parameters
$\alpha, 
\beta, \gamma$, $\delta_{\text in}$ and $\delta_{\text out}$, such
that $\alpha+\beta+\gamma=1$. To avoid degenerate situations
assume  each of the numbers $\alpha, 
\beta, \gamma$ is strictly smaller than 1. 

At each step of the growth algorithm we  add  one edge to an
existing graph to obtain a new graph, and we will enumerate the
obtained graphs by the number of edges they contain. 
Start with an  initial finite directed
graph, denoted $G(n_0)$,  with at least one node and $n_0$ edges.
 For $n =n_0+1,n_0+2,\ldots$, 
$G(n)$ will be a graph with $n$ edges and a random number $N(n)$ of
nodes. If $u$ is a node in $G(n-1)$, $D_{\rm in}(u)$ and $D_{\rm out}(u)$ 
denote the in  and out degree of $u$ respectively. The graph $G(n)$ is
obtained from the graph $G(n-1)$ as follows.
\begin{itemize}
\item 
With probability
$\alpha$  we append to $G(n-1)$ a new node $v$ and an edge leading
from $v$ to an existing node $w$ in $G(n-1)$ (denoted $v \mapsto w$).
The existing node $w$ in $G(n-1)$ is chosen with probability depending
on its in-degree:
$$
p(\text{$w$ is chosen}) = \frac{D_{\rm in}(w)+\delta_{\text
    in}}{n-1+\delta_{\text in}N(n-1)} \,.
$$
\item 
With probability $\beta$ we only append to $G(n-1)$ a directed edge
$v\mapsto w$
between two existing nodes $v$ and $w$ of $G(n-1)$.
 The existing nodes $v,w$ are chosen independently from the nodes of $G(n-1)$ with
 probabilities 
$$
p(\text{$v$ is chosen}) = \frac{D_{\rm out}(v)+\delta_{\text
    out}}{n-1+\delta_{\text out}N(n-1)}, \ \ 
p(\text{$w$ is chosen}) = \frac{D_{\rm in}(w)+\delta_{\text
    in}}{n-1+\delta_{\text in}N(n-1)}\,.
$$
\item With probability
$\gamma$  we append to $G(n-1)$ a new node $w$ and an edge $v\mapsto
w$ leading
from  the existing node $v$ in $G(n-1)$  to the new node $w$. The
existing node $v$ in $G(n-1)$ is chosen with probability
$$
p(\text{$v$ is chosen}) = \frac{D_{\rm out}(v)+\delta_{\text
    out}}{n-1+\delta_{\text out}N(n-1)}\,.
$$
\end{itemize}
If either
$\delta_{\text in}=0$, or $\delta_{\text out}=0$, we must have
$n_0\geq 1$ for the initial steps of the algorithm to make
sense. 

For $i,j=0,1,2,\ldots$ and $n\geq n_0$, let $N_{ij}(n)$ be the
(random) number of nodes in $G(n)$ with in-degree $i$ and out-degree
$j$. Theorem 3.2 in \cite{bollobas:borgs:chayes:riordan:2003}  shows
that there are nonrandom constants $(f_{ij})$ such that
\begin{equation} \label{e:limit.f}
\lim_{n\to\infty} \frac{N_{ij}(n)}{n}=f_{ij} \ \ \text{a.s. for
  $i,j=0,1,2,\ldots$.}
\end{equation}
Clearly, $f_{00}=0$. Since we obviously have 
$$
\lim_{n\to\infty} \frac{N(n)}{n}=1-\beta \ \ \text{a.s.,}
$$
we see that the empirical joint in- and out-degree distribution in the
sequence $\{G(n)\}$ of growing random graphs has as a nonrandom limit
the probability distribution 
\begin{equation} \label{e:limit.p}
\lim_{n\to\infty}
\frac{N_{ij}(n)}{N(n)}=\frac{f_{ij}}{1-\beta}:=p_{ij} \ \
\text{a.s. for   $i,j=0,1,2,\ldots$.}
\end{equation}
In \cite{bollobas:borgs:chayes:riordan:2003} it was shown that the
limiting degree distribution $(p_{ij})$ has, marginally, regularly
varying (in fact, power-like) tails. Specifically, Theorem 3.1 {\it
  ibid.} shows that for some finite positive constants $C_{\text in}$
and $C_{\text out}$ we have
\begin{equation} \label{e:marginal.regvar}
p_i(\text{in}):= \sum_{j=0}^\infty p_{ij} \sim C_{\text
  in}i^{-\alpha_{\text in}} \ \ \text{as $i\to\infty$, as long as
  $\alpha \delta_{\text in}+\gamma>0$,}  
\end{equation} 
$$
p_j(\text{out}):= \sum_{i=0}^\infty p_{ij} \sim C_{\text
  out}j^{-\alpha_{\text out}} \ \ \text{as $j\to\infty$, as long as
  $\gamma \delta_{\text out}+\alpha>0$.}
$$
Here
\begin{equation} \label{e:exponents}
\alphain = 1+ \frac{1+\deltain(\alpha+\gamma)}{\alpha+\beta}, \ \ 
\alphaout = 1+ \frac{1+\deltaout(\alpha+\gamma)}{\gamma+\beta}\,.
\end{equation}

In fact,  the limiting degree distribution $(p_{ij})$ in
\eqref{e:limit.p} generates a distribution that has jointly
nonstandard regularly
varying tails and the limit measure of regular variation has a density as shown in
\cite{resnick:samorodnitsky:towsley:davis:willis:wan:2014}. 

\subsection{Notation and results summary.}\label{subsec:results}
 We summarize results and
notation
for the
preferential attachment model from \cite{resnick:samorodnitsky:towsley:davis:willis:wan:2014}.  

\begin{align}
&c_1 = \frac{\alpha+\beta}{1+\deltain(\alpha+\gamma)}=\frac{1}{\alphain-1}, \quad
c_2 =
\frac{\beta+\gamma}{1+\deltaout(\alpha+\gamma)}, \label{eq:c1c2}
=\frac{1}{\alphaout-1}\\ 
&a=c_2/c_1.\label{eq:a}
\end{align}

We developed an explicit
formula for the joint generating function of in- and out-degree.
The joint generating function of $\{p_{ij}\}$ in   
\eqref{e:limit.p},  
\begin{equation} \label{e:phi}
\varphi(x,y) = \sum_{i=0}^\infty \sum_{j=0}^\infty x^i y^jp_{ij}, \
0\leq x,y\leq 1,
\end{equation}
satisfies a partial differential
equation that, when solved, yields 
\begin{equation} \label{e:split.gf}
\varphi(x,y) = \frac{\gamma}{\alpha+\gamma}  x \varphi_1(x,y) 
+ \frac{\alpha}{\alpha+\gamma}  y \varphi_2(x,y) \,,
\end{equation}
with
\begin{align} 
\varphi_1(x,y) =& c_1^{-1}\int_1^\infty 
z^{-(1+1/c_1)} \bigl( x+(1-x)z\bigr)^{-(\deltain+1)} \bigl(
y+(1-y)z^a\bigr)^{-\deltaout }\, dz\,, \label{e:phi1}\\
\varphi_2(x,y) =& c_1^{-1}\int_1^\infty 
z^{-(1+1/c_1)} \bigl( x+(1-x)z\bigr)^{-\deltain} \bigl(
y+(1-y)z^a\bigr)^{-(\deltaout +1)}\, dz \label{e:phi2}
\end{align}
for $0\leq x,y\leq 1$. Each of $\varphi_1,\varphi_2$ is
the joint generating function of a pair of nonnegative integer-valued
random variables;
that is,
on some probability space we can find nonnegative integer-valued
random variables $X_j,\, Y_j, \ j=1,2$ such that
$$
\varphi_j(x,y) = E\bigl( x^{X_j}y^{Y_j}\bigr), \ 0\leq x,y\leq 1, \
j=1,2\,.
$$
If $(I,O)$ is a random vector with generating function
\eqref{e:split.gf}, 
$$\varphi(x,y)=E\bigl( x^Iy^O\bigr),$$
we can represent the distribution of $(I,O)$ as
\begin{equation}\label{eq:io.rep}
(I,0)\stackrel{d}{=} B(1+X_1,Y_1)+(1-B) (X_2,1+Y_2),
\end{equation}
where $B$ is a Bernoulli switching variable independent of $X_j,Y_j,\,
j=1,2$ with
$$P[B=1]=1-P[B=0]=\frac{\gamma}{\alpha+\gamma}.$$

The  explicit structure and form in \eqref{e:split.gf},\eqref{e:phi1}
and \eqref{e:phi2} allowed analysis of the asymptotic multivariate
power law structure performed in
\cite{resnick:samorodnitsky:towsley:davis:willis:wan:2014}.  Absent such 
structure, if all one has is the joint generating function, one would have to rely on Tauberian analysis of the
transform. We show how the material in Section
\ref{sec:mult.reg.var} is applicable.

\subsection{Joint regular variation of the distribution of in-degree
and out-degree} \label{sec:InOutRegVar} 
In this section we apply  the Tauberian theorem of Section
\ref{sec:Tauberian.thm} to the joint generating
function $\varphi$ of the limiting distribution of  in- and 
out-degree given in 
 \eqref{e:split.gf},\eqref{e:phi1}
and \eqref{e:phi2} to prove the
nonstandard joint regular variation of  in- and 
out-degree and obtain an expression for the density of the tail
measure.

The next Theorem \ref{t:two.measures} shows that each  of the random
vectors $\bigl( X_j,\, Y_j\bigr)$, $j=1,2$, has a nonstandard
regularly varying distribution. The 
decomposition \eqref{e:split.gf} allows us to deduce the
nonstandard joint regular variation of $(I,O)$, the  in-degree and  out-degree. 

\begin{theorem} \label{t:two.measures} 
Assume that $\deltain>0$ and $\deltaout>0$, and 
let $\alphain$ and $\alphaout$ be given by
\eqref{e:exponents}. For each $j=1,2$ there is a Radon
measure $V_j \in M_+([0,\infty]^2\setminus \{ {\mathbf 0}\})$ such
that as $h\to\infty$,  
\begin{equation} \label{e:regvar.j}
hP\Bigl[ \bigl( {h^{-1/(\alphain-1)}}{X_j}, \,
h^{-1/(\alphaout-1)}Y_j\bigr)\in\cdot\,\Bigr]\vrightarrow V_j (\cdot),
\end{equation}
vaguely in $M_+([0,\infty]^2\setminus \{ {\mathbf 0}\})$. 
Furthermore, $V_1$  and $V_2$ concentrate on $(0,\infty)^2$ and
have Lebesgue densities $f_1,f_2$ 
 given by,
\begin{align} 
f_1(x,y) =& c_1^{-1} \bigl(
\Gamma(\deltain+1)\Gamma(\deltaout)\bigr)^{-1}
x^{\deltain}y^{\deltaout-1}
\int_0^\infty z^{-(2+1/c_1+\deltain +a\deltaout)} e^{-(x/z+y/z^a)}\,
dz ,\label{e:density} \\
\intertext{and}
f_2(x,y) =& c_1^{-1} \bigl(
\Gamma(\deltain)\Gamma(\deltaout+1)\bigr)^{-1}
x^{\deltain-1}y^{\deltaout}
\int_0^\infty z^{-(1+a+1/c_1+\deltain +a\deltaout)} e^{-(x/z+y/z^a)}\, dz\,.
\label{e:density.2}
\end{align}
The  random vector $\bigl( I,O)$ with joint mass function
 $\{p_{ij}\}$ in \eqref{e:limit.p} satisfies as $h\to\infty$,
$$
hP\Bigl[ \bigl( h^{-1/(\alphain-1)}I, \,
h^{-1/(\alphaout-1)}O\bigr)\in\cdot \,\Bigr] \vrightarrow 
\frac{\gamma}{\alpha+\gamma} V_1(\cdot)+ \frac{\alpha}{\alpha+\gamma}
V_2(\cdot), 
$$
 vaguely in $M_+([0,\infty]^2\setminus \{ {\mathbf 0}\})$. 
\end{theorem}
\begin{proof}
It is enough to prove \eqref{e:regvar.j}, \eqref{e:density} and
\eqref{e:density.2} .  
We treat the case $j=1$. The case $j=2$ is  analogous. 
Since $\varphi_1(x,y)$ is the generating function of a probability
mass function, simply converting $\varphi_1(x,y)$ into a Laplace
transform will not yield the Laplace transform of an infinite measure
$U$ as required by the previous section. So we first modify the generating function.

Choose and fix a positive integer $k>\alphain-1$. 
This choice of $k$ guarantees $E(X_1^k)=\infty$.
Denote
$$
\psi(x,y) = \frac{\partial^k \varphi_1}{\partial x^k}(x,y), \
0<x,y<1,
$$
so that  the function $\psi$  can be written in the form
\begin{equation}\label{e:defpsi}
\psi(x,y) = \sum_{i=0}^\infty \sum_{j=0}^\infty x^i y^j m^{(k)}_{ij}, 
\ 0< x,y< 1\,,
\end{equation}
where 
$$
m^{(k)}_{ij} = \prod_{d=1}^k (i+d) p^{(k)}_{ij}, \ i,j=0,1,2,\ldots\,,
$$
and $(p^{(k)}_{ij})$ is the joint probability mass function of the
random vector $(X_1-k,Y_1)$. Let %$m^{(k)}$ 
$U(\cdot)=\sum_{i,j} m^{(k)}_{ij}\epsilon_{(i,j)} (\cdot)$ 
be the infinite Radon measure on $[0,\infty)^2$ concentrating on
$\bigl( \{0,1,2,\ldots\}\bigr)^2$ that puts mass $m^{(k)}_{ij}$
at  $(i,j)$. To verify this measure is infinite, observe
$$\sum_{i,j} m_{ij}^{(k)}=\sum_{l=0}^\infty \prod_{p=1}^k
  (p+l)P[X_1=l+k]$$
and since 
$\prod_{p=1}^k(p+l) \sim (l+k)^k$ as $l\to \infty$ 
and $E(X_1^k)=\infty$, we have $\sum_{i,j} m_{ij}^{(k)}$ diverges.

Using Proposition \ref{prop:2},  we  show  that the measure $U$
is regularly varying:  As $h\to\infty$, we show,
\begin{equation} \label{e:regvar.m}
\frac1h U\left\{ (i,j): \, \bigl( h^{-1/(k-\alphain+1)}i, \,
  h^{-(\alphain-1)/((\alphaout-1)(k-\alphain+1))}j\bigr)\in \cdot\,\right\} 
\vrightarrow V_{1,k} (\cdot)
\end{equation}
vaguely in $M_+([0,\infty)^2)$,
where the Radon measure $V_{1,k}$  concentrates on $(0,\infty)^2$
with 
density 
\begin{equation} \label{e:density.k}
f_{1,k}(x,y) = c_1^{-1} \bigl(
\Gamma(\deltain+1)\Gamma(\deltaout)\bigr)^{-1}
x^{\deltain+k}y^{\deltaout-1}
\int_0^\infty z^{-(2+1/c_1+\deltain +a\deltaout)} e^{-(x/z+y/z^a)}\, dz\,.
\end{equation}
To this end, using the form of $\phi_1$ in \eqref{e:phi1},  we 
 write the function $\psi$ in
\eqref{e:defpsi} 
explicitly as
\begin{align*}
\psi(x,y) =& c_1^{-1} \prod_{i=1}^k  (\deltain+i) 
\int_1^\infty
z^{-(1+1/c_1)} (z-1)^k \bigl( x+(1-x)z\bigr)^{-(\deltain+k+1)} \bigl(
y+(1-y)z^a\bigr)^{-\deltaout }\, dz
\\
:=& c_1^{-1} \prod_{i=1}^k  (\deltain+i) \tilde \psi(x,y)\,.
\end{align*}
We switch from generating functions to Laplace transforms by replacing
$(x,y)$ with $e^{-\blambda}=(e^{-\lambda_1},e^{-\lambda_2})$ and then
consider regular variation of the resulting Laplace transform.
For fixed $\lambda_1>0, \, \lambda_2>0$ elementary calculations show
that, as $h\to\infty$, 
\begin{align*}
h^{-1} \tilde \psi &\left( e^{-\lambda_1 h^{-1/(k-\alphain+1)}}, \, e^{-\lambda_2
  h^{-(\alphain-1)/(\alphaout-1)(k-\alphain+1)}} \right)\\
&\sim h^{-1} \int_1^\infty z^{k-1-1/c_1}\left( 1+ z \lambda_1
  h^{-1/(k-\alphain+1)}\right)^{-(\deltain+k+1)}  \\
&\qquad \qquad
\times \left( 1+ z^a \lambda_2
  h^{-(\alphain-1)/((\alphaout-1)(k-\alphain+1))}\right)^{-\deltaout }\, dz\\
&= \int_{h^{-1/(k-\alphain+1)}}^\infty z^{k-1-1/c_1}\left( 1+ z
  \lambda_1\right)^{-(\deltain+k+1)}  \left( 1+ z^a \lambda_2\right)^{-\deltaout }\, dz
\\
&\to 
\int_{0}^\infty z^{k-1-1/c_1}\left( 1+ z
  \lambda_1\right)^{-(\deltain+k+1)}  \left( 1+ z^a
  \lambda_2\right)^{-\deltaout }\, dz\,. 
\end{align*}
We conclude that for any $\lambda_1>0, \, \lambda_2>0$, as $h\to\infty$, 
\begin{align}
h^{-1} \hat U&(\lambda_1  h^{-1/(k-\alphain+1)} ,  \lambda_2      h^{-(\alphain-1)/((\alphaout-1)(k-\alphain+1))})
 \label{e:uhat}\\
&=h^{-1}  \psi \left( e^{-\lambda_1 h^{-1/(k-\alphain+1)}}, \, e^{-\lambda_2
  h^{-(\alphain-1)/((\alphaout-1)(k-\alphain+1))}}    \right)\nonumber\\
&\to c_1^{-1} \prod_{i=1}^k  (\deltain+i) \int_{0}^\infty z^{k-1-1/c_1}\left( 1+ z
  \lambda_1\right)^{-(\deltain+k+1)}  \left( 1+ z^a
  \lambda_2\right)^{-\deltaout }\, dz
\nonumber\\
&= \int_{[0,\infty)^2} e^{-(\lambda_1v_1+\lambda_2v_2)}\,
V_{1,k}(dv_1,dv_2)\,,\nonumber
\end{align}
where the measure $V_{1,k}$ concentrates on $(0,\infty)^2$
and has  density
$$ f_{1,k}(x,y)=
 c_1^{-1} \prod_{i=1}^k  (\deltain+i) 
\int_{0}^\infty z^{k-1-1/c_1} \frac{x^{\deltain+k}z
  ^{-(\deltain+k+1)}}{\Gamma(\deltain+k+1)} e^{-x/z}
\frac{y^{\deltaout-1}(z^a)^{-\deltaout}}{\Gamma(\deltaout)}e^{-y/z^a}
dz,$$
given by \eqref{e:density.k}. 

The claim \eqref{e:regvar.m} now follows from \eqref{e:uhat} and the
Tauberian result in Proposition \ref{prop:2} provided we check that the measure $U$
satisfies condition \eqref{eq:7early} of that result so we must check
with
$$\bb(h)=\bigl(h^{1/(k-\alphain+1)}, \,   h^{(\alphain-1)/((\alphaout-1)(k-\alphain+1))} \bigr)
$$that
\begin{equation}\label{e:pleaseBeTrue}
\lim_{y\to\infty}\limsup_{h\to\infty} \int_{[v_1>y]\cup [v_2>y]
}e^{-\blambda' \bv }\, h^{-1}U(\bb (h) d\bv )=0.
\end{equation}
Considering the definition of $U(\cdot)$ the integral in
\eqref{e:pleaseBeTrue} becomes, after a 
change of variable $s_i=b_i(h)v_i$,
\begin{align*}
\int_{[s_1>b_1(h)y]\cup [s_2>b_2(h) y]
}&e^{-(\lambda_1 s_1/b_1(h)+\lambda_2 s_2/b_2(h) )  }\, h^{-1}U(d\bs )
\\
=&\sum_{[i>b_1(h)y]\cup [j>b_2(h) y]}
e^{-(\lambda_1 i/b_1(h)+\lambda_2 j/b_2(h) )  }\,  h^{-1} m_{ij}^{(k)}\\
=&\sum_{[i>b_1(h)y]\cup [j>b_2(h) y]}
e^{-(\lambda_1 i/b_1(h)+\lambda_2 j/b_2(h) )  }\, h^{-1} 
\prod_{d=1}^k (i+d) p^{(k)}_{ij}\\
\leq & \sum_{i>b_1(h)y}   + \sum_{j>b_2(h)y}   .
\end{align*}
Notice that
\begin{align*}
& \sum_{i>b_1(h)y} e^{-(\lambda_1 i/b_1(h)+\lambda_2 j/b_2(h) )  }\, h^{-1} 
\prod_{d=1}^k (i+d) p^{(k)}_{ij}\\ 
\leq &\sum_{i>b_1(h)y}
e^{-(\lambda_1 i/b_1(h) )  }\, h^{-1} 
\sum_j\prod_{d=1}^k (i+d) p^{(k)}_{ij}\\
=&\sum_{i>b_1(h)y}
e^{-(\lambda_1 i/b_1(h) )  }\, h^{-1} 
\prod_{d=1}^k (i+d) p_{i+k}(in)\\
\intertext{using the notation from \eqref{e:marginal.regvar}. Set
  $u_i=
\prod_{d=1}^k (i+d) p_{i+k}(in)$  so from \eqref{e:marginal.regvar} $u_i
\sim  C_{\text{in}} i^{k-\alphain}$. Letting  $C$ be a finite
constant, the sum on the previous line is
bounded by   
}
& C \sum_{i>b_1(h)y}
e^{-(\lambda_1 i/b_1(h) )  }\, h^{-1} i^{k-\alphain} \\
\sim & C  h^{-1}\int_{b_1(h)y}^\infty e^{-(\lambda_1 x/b_1(h) )
}x^{k-\alphain} \, dx \\
\to & C  \int_{y}^\infty e^{-\lambda_1 x}x^{k-\alphain} \, dx 
,\qquad (h\to\infty),\\
\to & 0 \qquad (y\to\infty).
\end{align*}
In particular, given $\vep>0$, there is $\theta_\vep\in (0,\infty)$
such that for all $h$ large enough,
$$
\sum_{i>b_1(h)\theta_\vep} e^{-(\lambda_1 i/b_1(h)+\lambda_2 j/b_2(h)
  )  }\, h^{-1}  \prod_{d=1}^k (i+d) p^{(k)}_{ij}\leq \vep.
$$
For such $h$,
$$
 \sum_{j>b_2(h)y} \leq \vep + \sum_{i\leq b_1(h)\theta_\vep, \,
   j>b_2(h)y} .
$$
Further, for some positive constant $C$, 
\begin{align*}
& \sum_{i\leq b_1(h)\theta_\vep, \,   j>b_2(h)y} e^{-(\lambda_1
  i/b_1(h)+\lambda_2 j/b_2(h) )  }\, h^{-1}  
\prod_{d=1}^k (i+d) p^{(k)}_{ij}\\ 
\leq & C \sum_{i\leq b_1(h)\theta_\vep, \,   j>b_2(h)y} h^{-1}  
i^k p^{(k)}_{ij}\\ 
\leq & C \theta_\vep^k h^{-1}   b_1(h)^k  \sum_{j>b_2(h)y}
p^{(k)}_{ij}\\ 
= & C \theta_\vep^k h^{(\alphain-1)/(k-\alphain+1)} \sum_{j>b_2(h)y}
p_j(out) \\
\sim & \bigl( C C_{out}\theta_\vep^k/(\alphaout-1)\bigr)
h^{(\alphain-1)/(k-\alphain+1)} (b_2(h)y)^{-(\alphaout-1)} \\
= & \bigl( C C_{out}\theta_\vep^k/(\alphaout-1)\bigr)
y^{-(\alphaout-1)} \\
\to & 0 \qquad (y\to\infty)
\end{align*}
by the Karamata theorem, using the notation from
\eqref{e:marginal.regvar}.  

Letting $\vep\to 0$ we see that we have verified that the measure $U$
satisfies condition \eqref{eq:7early} and that \eqref{e:regvar.m}
holds 
and we are now ready to prove \eqref{e:regvar.j}. 
Let $\mu^{(k)}=\sum_{i,j} p_{ij}^{(k)}\epsilon_{(i,j)}$ be the
probability measure concentrating on
$\bigl( \{0,1,2,\ldots\}\bigr)^2$ that puts mass  $p^{(k)}_{ij}$
at $(i,j)$. For \eqref{e:regvar.j},
 it is enough
to prove that for any $a,b>0$,
\begin{equation} \label{e:m.to.nu}
h\int_{h^{1/(\alphain-1)}a}^\infty \, 
\int_{h^{1/(\alphaout-1)}b}^\infty 
\mu^{(k)}(dx,dy) \to 
\int_a^\infty\int_b^\infty f_1(x,y)\, dxdy
\end{equation}
as $h\to\infty$, with $f_1$ given by \eqref{e:density}.
 Indeed, by Theorem 3.2 in
\cite{bollobas:borgs:chayes:riordan:2003}, the conditional distributions
of the random vector $\bigl( I,O)$ are also regular varying with
exponents of regular variation strictly larger than those of the
marginal distributions. Therefore, one can trivially add the axes
$\{x=0, \, y>0\}$ and $\{ x>0, \, y=0\}$ to the convergence in
\eqref{e:m.to.nu} which yields
$$
hP\Bigl[ \bigl( h^{-1/(\alphain-1)}(X_1-k), \,
h^{-1/(\alphaout-1)}Y_1\bigr)\in\cdot \,\Bigr] \vrightarrow V_1(\cdot)\,,
$$
which is equivalent to \eqref{e:regvar.j} with $j=1$. 

It remains, therefore, to prove \eqref{e:m.to.nu}. Fix
$M>\max(a,b)$. 
Since
$$\mu^{(k)}(dx,dy)=\frac{U(dx,dy)}{\prod_{d=1}^k (x+d)},$$
we have, as $h\to\infty$,
\begin{align*}
h\int_{h^{1/(\alphain-1)}a}^{h^{1/(\alphain-1)}M} &
\, 
\int_{h^{1/(\alphaout-1)}b}^{h^{1/(\alphaout-1)}M}
\mu^{(k)}(dx,dy)
\sim h\int_{h^{1/(\alphain-1)}a}^{h^{1/(\alphain-1)}M}
\, 
\int_{h^{1/(\alphaout-1)}b}^{h^{1/(\alphaout-1)}M}
x^{-k}U(dx,dy)\\
=& h^{1-k/(\alphaout-1)}\int_a^M\int_b^M x^{-k} U\bigl(
dh^{1/(\alphain-1)}x, \, dh^{1/(\alphaout-1)}y\bigr)\,.\\
\intertext{Denoting $m_h =  h^{k/(\alphaout-1)-1}$, we can 
write the above as}
=&\frac{1}{m_h} \int_a^M\int_b^M x^{-k} U \bigl(
m_h^{1/(k-\alphain+1)}dx, \,
m_h^{(\alphain-1)/((\alphaout-1)(k-\alphain+1))} dy\bigr)
\\
\to & \int_a^M\int_b^M x^{-k}f_{1,k}(x,y)\, dx\, dy
\end{align*}
as $h\to\infty$ by  \eqref{e:regvar.m}. Since 
$$
f_1(x,y) = x^{-k}f_{1,k}(x,y), \ 0<x,y<1\,,
$$
the statement \eqref{e:regvar.j} with $j=1$ follows, because by
\eqref{e:marginal.regvar} and \eqref{e:split.gf}, 
\begin{align*}
\limsup_{h\to\infty} 
& \,h\left[ \int_{h^{1/(\alphain-1)}M}^\infty \int_0^\infty
  \mu^{(k)}(dx,dy)
+ \int_0^\infty\int_{h^{1/(\alphaout-1)}M} \mu^{(k)}(dx,dy)\right]\\
\leq & \limsup_{h\to\infty} hP\bigl( X_1>h^{1/(\alphain-1)}M +k\bigr) + 
\limsup_{h\to\infty}hP\bigl( Y_1>h^{1/(\alphaout-1)}M\bigr)\\
\leq &\frac{\alpha+\gamma}{\gamma} \frac{C_{\text in}}{\alphain-1}
M^{-(\alphain-1)} + \frac{\alpha+\gamma}{\alpha} \frac{C_{\text out}}{\alphaout-1}
M^{-(\alphaout-1)}\,,
\end{align*}
and one only needs to let $M\to\infty$. 

As mentioned before, the case of \eqref{e:regvar.j} with  $j=2$ is
analogous. 
\end{proof}

\section{Acknowledgment}
We appreciate several helpful and informative conversations with Don
Towsley and Bo Jiang of the University of Massachusetts.

\bibliographystyle{Genamystyle}
\bibliography{Genabibfile}

\begin{thebibliography}{12}
\expandafter\ifx\csname natexlab\endcsname\relax\def\natexlab#1{#1}\fi

\bibitem[Bingham et~al.(1987)Bingham, Goldie and
  Teugels]{bingham:goldie:teugels:1987}
{\sc N.~Bingham, C.~Goldie {\rm and} J.~Teugels} (1987): {\em Regular
  Variation\/}.
\newblock Cambridge University Press, Cambridge.

\bibitem[Bollob\'as et~al.(2003)Bollob\'as, Borgs, Chayes and
  Riordan]{bollobas:borgs:chayes:riordan:2003}
{\sc B.~Bollob\'as, C.~Borgs, J.~Chayes {\rm and} O.~Riordan} (2003): Directed
  scale-free graphs.
\newblock In {\em Proceedings of the Fourteenth Annual ACM-SIAM Symposium on
  Discrete Algorithms (Baltimore, 2003)\/}. ACM, New York, pp. 132--139.

\bibitem[Feller(1971)]{feller:1971}
{\sc W.~Feller} (1971): {\em An Introduction to Probability Theory and its
  Applications\/}, volume~2.
\newblock Wiley, New York, 2nd edition.

\bibitem[Krapivsky and Redner(2001)]{krapivsky:redner:2001}
{\sc P.~Krapivsky {\rm and} S.~Redner} (2001): Organization of growing random
  networks.
\newblock {\em Physical Review E\/} 63:066123:1--14.

\bibitem[Resnick(1991)]{resnick:1991}
{\sc S.~Resnick} (1991): Point processes and {T}auberian theory.
\newblock {\em Math. Sci.\/} 16:83--106.

\bibitem[Resnick(2007)]{resnick:2007}
{\sc S.~Resnick} (2007): {\em Heavy-Tail Phenomena: Probabilistic and
  Statistical Modeling\/}.
\newblock Springer, New York.

\bibitem[{Samorodnitsky} et~al.(2014){Samorodnitsky}, {Resnick}, {Towsley},
  {Davis}, {Willis} and
  {Wan}]{resnick:samorodnitsky:towsley:davis:willis:wan:2014}
{\sc G.~{Samorodnitsky}, S.~{Resnick}, D.~{Towsley}, R.~{Davis}, A.~{Willis}
  {\rm and} P.~{Wan}} (2014): {Nonstandard regular variation of in-degree and
  out-degree in the preferential attachment model}.
\newblock {\em ArXiv e-prints\/} \url{http://arxiv.org/pdf/1405.4882.pdf}.

\bibitem[Stadtm{\"u}ller(1981)]{stadtmuller:1981}
{\sc U.~Stadtm{\"u}ller} (1981): A refined {T}auberian theorem for {L}aplace
  transforms in dimension {$d>1$}.
\newblock {\em J. Reine Angew. Math.\/} 328:72--83.

\bibitem[Stadtm{\"u}ller and Trautner(1979)]{stadtmuller:trautner:1979}
{\sc U.~Stadtm{\"u}ller {\rm and} R.~Trautner} (1979): Tauberian theorems for
  {L}aplace transforms.
\newblock {\em J. Reine Angew. Math.\/} 311/312:283--290.

\bibitem[Stadtm{\"u}ller and Trautner(1981)]{stadtmuller:trautner:1981}
{\sc U.~Stadtm{\"u}ller {\rm and} R.~Trautner} (1981): Tauberian theorems for
  {L}aplace transforms in dimension {$D>1$}.
\newblock {\em J. Reine Angew. Math.\/} 323:127--138.

\bibitem[Stam(1977)]{stam:1977}
{\sc A.~Stam} (1977): Regular variation in $\mathbb{R}_+^d$ and the
  {A}bel-{T}auber theorem.
\newblock Technical Report, unpublished, Mathematisch Instituut,
  Rijksuniversiteit Groningen.

\bibitem[Yakimiv(2005)]{yakimiv:2005}
{\sc A.~Yakimiv} (2005): {\em Probabilistic Applications of Tauberian
  Theorems\/}.
\newblock Modern Probability and Statistics. VSP, Leiden, The Netherlands.

\end{thebibliography}

\end{document}